\documentclass[11pt, reqno]{amsart}
\usepackage{amssymb,latexsym,amsmath,amsfonts,mathdots,enumitem,amsxtra}
\usepackage{latexsym}
\usepackage[mathscr]{eucal}
\usepackage{colortbl,xcolor}
\usepackage{lmodern}
\usepackage{sansmathaccent}
\usepackage[latin1]{inputenc}
\usepackage{tikz}
\usepackage{physics}
\usepackage{tikz-cd}
\usepackage{graphicx}
\usepackage{subcaption}
\usepackage{url}
\usetikzlibrary{shapes,arrows}
\usetikzlibrary{matrix,calc,shapes,arrows,positioning}
\pdfmapfile{+sansmathaccent.map}

\numberwithin{equation}{section}
\theoremstyle{plain}
\newtheorem{exam}{Example}

\newtheorem{thm}{Theorem}[section]

\newtheorem{lem}[thm]{Lemma}
\newtheorem{prop}[thm]{Proposition}
\theoremstyle{definition}

\newtheorem{rem}[thm]{Remark}

\numberwithin{equation}{section}

\def\beq{\begin{eqnarray}}
	\def\eeq{\end{eqnarray}}
\def\beqa{\begin{eqnarray*}}
	\def\eeqa{\end{eqnarray*}}

\def\beqn{\begin{equation}}
	\def\eeqn{\end{equation}}

\def\mg#1{}

\renewcommand{\epsilon}{\varepsilon}
\renewcommand{\phi}{\varphi}

\renewcommand{\bf}[1]{\textbf{#1}}
\renewcommand{\it}[1]{\textit{#1}}

\renewcommand{\sf}[1]{\textsf{#1}}

\numberwithin{equation}{section}
\allowdisplaybreaks[4] 

\setlist[enumerate]{font=\upshape,noitemsep, topsep=0pt} 
\setlist[itemize]{noitemsep, topsep=0pt}

\begin{document}
	
\title[Revisiting the Schwarz Lemma for the Pentablock]
{Revisiting the Schwarz Lemma for the Pentablock: Counterexamples and Sharp Interpolation Criteria}
\author{Shubhankar Mandal, Avijit Pal and Bhaskar Paul}
	
\address[S. Mandal]{Department of Mathematics, IIT Bhilai, 6th Lane Road, Jevra, Chhattisgarh 491002}
\email{S. Mandal:shubho0204@gmail.com}

\address[A. Pal]{Department of Mathematics, IIT Bhilai, 6th Lane Road, Jevra, Chhattisgarh 491002}
\email{A. Pal:avijit@iitbhilai.ac.in}

\address[B. Paul]{Department of Mathematics, IIT Bhilai, 6th Lane Road, Jevra, Chhattisgarh 491002}
\email{B. Paul:bhaskarpaul@iitbhilai.ac.in }

\begin{abstract}
Alshehri and Lykova (\emph{J. Geom. Anal.}
\textbf{33} (2023)) formulated a Schwarz lemma
for the pentablock $P$, giving a criterion for
the existence of an analytic map from the unit
disc $\mathbb D$ into the closed pentablock
$\overline P$ that sends the origin to the origin
and a prescribed point $\lambda_0\in\mathbb D$
to a given point $x_0\in\overline P$.
We show that the sufficiency direction of their
criterion is false. We identify the
underlying error as the extension to the open
unit disc of a Fej\'er--Riesz identity that is
valid only on the unit circle.
We establish a corrected necessary and
sufficient interpolation criterion when the
last two components are prescribed by a fixed
rational $\Gamma$-inner function and resolve
the special case $s_0=0$. Using the Poisson
integral representation of the associated
outer spectral factor, we obtain a sharp
comparison between the corrected bound and
the bound proposed by Alshehri and Lykova,
with a precise characterization of equality.
We further express the logarithmic difference
between these bounds as a nonnegative Green
potential, thereby quantifying the discrepancy
in terms of the derivative of the first
coordinate of the prescribed $\Gamma$-inner
function.
Finally, we formulate a variational problem
over rational $\Gamma$-inner interpolants
satisfying the prescribed two-point conditions.
We show that the supremum of the associated
outer-factor quantities determines the
maximal admissible modulus of the first
coordinate among pentablock-valued
interpolants whose last two components belong
to this class. Equivalently, this extremal
problem reduces to minimizing the corresponding
Green potential. We conclude by discussing
the remaining problem of characterizing
unrestricted Schwarz interpolation for the
pentablock and its connection with the
Lempert function.
\end{abstract}
\maketitle

\section{Introduction}
\label{sec:intro}

The pentablock is the domain
\[
P=
\left\{
(a_{21},\operatorname{tr}A,\det A):
A\in\mathbb C^{2\times2},\ \|A\|<1
\right\}
\subset\mathbb C^3.
\]
It was introduced by Agler, Lykova, and Young
\cite{ALY15} in connection with a problem in
$\mu$-synthesis. Since then, the complex geometry
of the pentablock and its associated function
theory have attracted considerable attention.
In particular, the study of holomorphic mappings
into the pentablock is closely related to
Schwarz-type interpolation problems, invariant
distances, and extremal problems in several
complex variables.
The pentablock is naturally related to the
symmetrized bidisc
\[
\mathbb G
=
\left\{
(z_1+z_2,z_1z_2):
z_1,z_2\in\mathbb D
\right\},
\]
whose closure is denoted by $\Gamma$.
Indeed, the projection
\[
\pi:P\longrightarrow\mathbb G,
\qquad
\pi(a,s,p)=(s,p),
\]
exhibits the pentablock as a domain fibred
over the symmetrized bidisc. Consequently,
the geometry of the fibres and the function
theory of $\Gamma$ play an important role
in understanding holomorphic interpolation
on the pentablock.
A fundamental problem in this setting is
to characterize the existence of an analytic
map
\(
x=(a,s,p):
\mathbb D\longrightarrow\overline P
\)
satisfying the two-point interpolation
conditions
\[
x(0)=(0,0,0),
\quad
x(\lambda_0)=(a_0,s_0,p_0),
\]
where $\lambda_0\in\mathbb D\setminus\{0\}$
and $(a_0,s_0,p_0)\in\overline P$.
This problem is commonly referred to as
the Schwarz interpolation problem for
the pentablock. Its solution is closely
connected with the determination of the
Lempert function $\widetilde{k}_P(0,\cdot)$
and with the question of whether the
Lempert function coincides with the
Carath\'eodory distance on $P$.
Agler, Lykova, and Young \cite{ALY15}
raised this question and identified the
determination of a Schwarz lemma for the
pentablock as an important problem in
its function theory.
Alshehri and Lykova \cite{AL23} studied
rational maps from the unit disc into
the closed pentablock whose boundary
values lie in the distinguished boundary
of $P$. Such maps are called rational
$\overline P$-inner functions. Using the
structure of these functions, they
formulated a Schwarz lemma for the
pentablock and proposed an interpolation
criterion. More precisely, the criterion
appears as Theorem~1.5 of \cite{AL23}
and is reproduced as Theorem~11.3,
with related explicit formulas given
in Theorem~11.2. The latter numbering
refers to the version of \cite{AL23}.
The criterion combines conditions
arising from the Schwarz lemma for
the symmetrized bidisc with an
additional restriction on the first
coordinate of the interpolation target.
The sufficiency argument relies on
the structure of rational
$\Gamma$-inner functions and their
associated Fej\'er--Riesz
factorizations. This raises the
question of whether the boundary
identities involved in the
factorization provide the required
estimates at interior interpolation
points.
The first objective of the present
paper is to examine the validity
of the sufficiency direction of
the interpolation criterion in
\cite{AL23}. We show that this
direction is false. More precisely,
for every
$s_0\in\mathbb D\setminus\{0\}$,
we construct a boundary point
\(
x_0=(a_0,s_0,0)
\in\partial\overline P
\)
and a parameter
$\lambda_0\in\mathbb D\setminus\{0\}$
for which all the proposed
interpolation conditions are
satisfied, although no analytic
map
\(
x:\mathbb D\longrightarrow\overline P
\)
can satisfy
\[
x(0)=(0,0,0),
\quad
x(\lambda_0)=x_0.
\]
Thus the stated conditions are
not sufficient for unrestricted
Schwarz interpolation on the
pentablock.
We also identify the precise
source of the error in the
original argument. Let
\[
h=(s,p)
=
\left(
\frac{E}{D},
\frac{D^{\sim n}}{D}
\right)
\]
be a rational $\Gamma$-inner
function, and let $A$ be the
associated outer polynomial
obtained through Fej\'er--Riesz
factorization. The factorization
gives
\[
\left|
\frac{A(\zeta)}{D(\zeta)}
\right|^2
=
1-\frac14|s(\zeta)|^2,
\quad \zeta\in\mathbb T.
\]
However, this identity holds
on the unit circle and does
not, in general, extend to
the open unit disc. In
particular, it cannot be
evaluated directly at an
interior interpolation point
without further justification.
The use of this boundary
identity as an interior
identity is the step
responsible for the failure
of the claimed sufficiency.
Having identified this
obstruction, we establish
a corrected necessary and
sufficient interpolation
criterion when the last
two components are fixed
to be a prescribed rational
$\Gamma$-inner function.
More precisely, suppose that
\[
h(0)=(0,0),
\quad
h(\lambda_0)=(s_0,p_0).
\]
We prove that there exists
an analytic map
\(
x=(a,s,p):
\mathbb D\longrightarrow\overline P
\)
with prescribed last two
components $h$ and satisfying
\[
x(0)=(0,0,0),
\quad
x(\lambda_0)=(a_0,s_0,p_0)
\]
if and only if
\[
|a_0|
\leq
|\lambda_0|
\left|
\frac{A(\lambda_0)}
{D(\lambda_0)}
\right|.
\]
Moreover, whenever this
inequality holds, the
interpolant can be chosen
to be a rational
$\overline P$-inner function.
The necessity follows from
the boundary geometry of
the pentablock, the outer
factorization, the Smirnov
maximum principle, and the
classical Schwarz lemma.
The sufficiency is obtained
through a suitable finite
Blaschke product and the
rational pentablock-inner
construction of \cite{AL23}.
We also resolve the
interpolation problem in
the special case $s_0=0$.

The second objective of
this paper is to clarify
the relationship between
the corrected interpolation
bound and the bound
appearing in \cite{AL23}.
The Poisson integral formula
plays a central role in
this comparison. Since
$A/D$ is outer, its
interior modulus is
determined by its boundary
modulus through the
Poisson representation
\[
\begin{aligned}
2\log\left|
\frac{A(\lambda_0)}
{D(\lambda_0)}
\right|
=
\int_{\mathbb T}
P_{\lambda_0}(\zeta)
\log\left(
1-\frac14|s(\zeta)|^2
\right)
\,dm(\zeta),
\end{aligned}
\]
where $P_{\lambda_0}$
denotes the Poisson kernel
and $m$ is normalized
Lebesgue measure on
$\mathbb T$.
Using the superharmonicity
of
\[
\lambda\longmapsto
\log\left(
1-\frac14|s(\lambda)|^2
\right),
\]
we establish the sharp
comparison
\[
\left|
\frac{A(\lambda_0)}
{D(\lambda_0)}
\right|
\leq
\sqrt{1-\frac14|s_0|^2}.
\]
Equality holds if and only
if $s$ is constant.
Consequently, the
inequality is strict
whenever $s$ is
nonconstant. This
comparison explains
why the pointwise
bound proposed in
\cite{AL23} is generally
larger than the exact
outer-factor bound
arising from a fixed
rational $\Gamma$-inner
interpolant.

We further refine this
comparison by expressing
the difference between
the logarithms of the
two bounds as a Green
potential. More precisely,
we prove that
\[
\begin{aligned}
\frac12
\log\left(
1-\frac14|s_0|^2
\right)
-
\log\left|
\frac{A(\lambda_0)}
{D(\lambda_0)}
\right|
=
\frac{1}{4\pi}
\int_{\mathbb D}
G(\lambda_0,\lambda)
\frac{|s'(\lambda)|^2}
{\left(
1-\frac14|s(\lambda)|^2
\right)^2}
\,dA(\lambda),
\end{aligned}
\]
where $G$ is the Green
function of the unit
disc and $dA$ denotes
planar Lebesgue area
measure. This identity
provides an exact
quantitative description
of the discrepancy
between the two bounds.
In particular, the
difference is governed
by the Green potential
of a nonnegative
density determined
by the derivative
of $s$.

The third objective
is to formulate an
associated variational
interpolation problem.
For fixed
$\lambda_0\in
\mathbb D\setminus\{0\}$
and $(s_0,p_0)\in\Gamma$,
let
\[
\begin{aligned}
\mathcal H(\lambda_0;s_0,p_0)
=
\bigl\{
h=(s,p):
h\text{ is rational }
\Gamma\text{-inner},~h(0)=(0,0),\
h(\lambda_0)=(s_0,p_0)
\bigr\}.
\end{aligned}
\]
Whenever this class
is nonempty, define
\[
M(\lambda_0;s_0,p_0)
=
\sup_{h\in
\mathcal H(\lambda_0;s_0,p_0)}
\left|
\frac{A_h(\lambda_0)}
{D_h(\lambda_0)}
\right|,
\]
where $A_h$ and $D_h$
are the associated
Fej\'er--Riesz factors.
We show that
\(
|\lambda_0|
M(\lambda_0;s_0,p_0)
\)
is precisely the
supremum of the
admissible values
of $|a_0|$ among
analytic pentablock-valued
interpolants whose
last two components
belong to
$\mathcal H(\lambda_0;s_0,p_0)$.
Furthermore, the
Green-potential
representation gives
\[
\begin{aligned}
M(\lambda_0;s_0,p_0)
=
\sqrt{1-\frac14|s_0|^2}
\exp\left(
-\inf_{h=(s,p)\in
\mathcal H(\lambda_0;s_0,p_0)}
\mathcal G_s(\lambda_0)
\right).
\end{aligned}
\]
Thus the maximization
of the outer-factor
quantity is equivalent
to the minimization
of the corresponding
Green potential over
the admissible rational
$\Gamma$-inner class.
The variational
characterization
obtained in this
way is exact within
the prescribed class
of rational
$\Gamma$-inner
interpolants.
However, it does
not automatically
extend to arbitrary
analytic maps into
the pentablock,
since the last two
components of such
maps need not be
rational $\Gamma$-inner.
Accordingly, a
complete characterization
of unrestricted
Schwarz interpolation
on the pentablock
remains a separate
problem. Such a
characterization
would provide further
information about
the Lempert function
of the pentablock
and its relationship
with the Carath\'eodory
distance.

The paper is organized
as follows. In
Section~\ref{sec:counter},
we construct counterexamples
to the sufficiency
direction of the
criterion proposed
in \cite{AL23}.
In Section~\ref{sec:where},
we examine the original
argument and identify
the precise step
responsible for its
failure. In
Section~\ref{sec:corrected},
we establish the
corrected interpolation
criterion for a fixed
rational $\Gamma$-inner
pair and resolve the
special case $s_0=0$.
We then derive the
Poisson representation
of the associated
outer factor and
obtain a sharp
comparison with the
proposed bound.
This comparison is
subsequently refined
through a Green-potential
representation.
Finally, we formulate
the corresponding
variational problem
over rational
$\Gamma$-inner
interpolants and
discuss the remaining
questions concerning
unrestricted Schwarz
interpolation on
the pentablock.
\section{Preliminaries}
\label{sec:prelim}

Let \(\mathbb D\) denote the open unit disc and let
\(\mathbb T=\partial\mathbb D\) be the unit circle. The
\emph{symmetrized bidisc} (see \cite{JAY, JAgler, AglerM, AglerMY, AY01, costara}) is the domain
\[
\mathbb G
=
\{(z+w,zw):z,w\in\mathbb D\},
\]
and we denote its closure by \(\Gamma=\overline{\mathbb G}\).
If \((a,s,p)\in\overline P\), then \((s,p)\in\Gamma\), since the
eigenvalues of a matrix in the closed unit ball have modulus at most
one.
For \((s,p)\in\Gamma\), write
\[
(s,p)=(\lambda_1+\lambda_2,\lambda_1\lambda_2),
\qquad \lambda_1,\lambda_2\in\overline{\mathbb D}.
\]
Define
\[
F(s,p)
=
\frac12\left|1-\overline{\lambda_2}\lambda_1\right|
+
\frac12
\sqrt{1-|\lambda_1|^2}
\sqrt{1-|\lambda_2|^2}.
\]
The value of \(F(s,p)\) is independent of the labelling of
\(\lambda_1\) and \(\lambda_2\). For each \(z\in\mathbb D\), observe
that
\[
1-sz+pz^2=(1-\lambda_1z)(1-\lambda_2z)\neq0.
\]
Consequently, the function
\[
\Psi_z(a,s,p)
=
\frac{a(1-|z|^2)}{1-sz+pz^2}
\]
is well-defined for \((a,s,p)\in\mathbb C^3\) whenever
\((s,p)\in\Gamma\).
We recall the following characterization of the closed pentablock.
\begin{prop}[{\cite[Theorem~5.3]{ALY15}}]
\label{fact:membership}
Let \((s,p)\in\Gamma\) and \(a\in\mathbb C\). The following
conditions are equivalent:
\begin{enumerate}
    \item \((a,s,p)\in\overline P\);
    \item \(|a|\leq F(s,p)\);
    \item \(|\Psi_z(a,s,p)|\leq1\) for every \(z\in\mathbb D\).
\end{enumerate}
\end{prop}
The function \(F\) has a particularly simple expression on the
distinguished boundary of the symmetrized bidisc.
\begin{prop}
\label{fact:boundary}
If \((s,p)\in\Gamma\) and \(|p|=1\), then
\[
F(s,p)=\sqrt{1-\frac14|s|^2}.
\]
In particular, this formula holds for every
\((s,p)\in b\Gamma\), the distinguished boundary of \(\Gamma\).
\end{prop}
\begin{proof}
Write
\(
s=\lambda_1+\lambda_2, p=\lambda_1\lambda_2,
\)
where \(\lambda_1,\lambda_2\in\overline{\mathbb D}\). Since
\(|p|=1\), we have
\(
|\lambda_1|=|\lambda_2|=1.
\)
It follows that
\[
F(s,p)=\frac12|1-\overline{\lambda_2}\lambda_1|.
\]
Moreover, we have
\[
\begin{aligned}
1-\frac14|s|^2
&=1-\frac14|\lambda_1+\lambda_2|^2\\
&=\frac12\bigl(1-\operatorname{Re}
  (\lambda_1\overline{\lambda_2})\bigr)\\
&=\frac14|1-\lambda_1\overline{\lambda_2}|^2.
\end{aligned}
\]
Taking square roots yields the desired identity.
\end{proof}
The next lemma identifies the supremum of the functions
\(\Psi_z\) in terms of \(F\).
\begin{lem}
\label{lem:sup}
Let \((s,p)\in\Gamma\) and suppose that \(F(s,p)>0\). Then, for
every \(a\in\mathbb C\),
\[
\sup_{z\in\mathbb D}|\Psi_z(a,s,p)|
=
\frac{|a|}{F(s,p)}.
\]
\end{lem}
\begin{proof}
For fixed \((s,p)\), the function \(\Psi_z(a,s,p)\) is linear in
\(a\). Hence
\[
\sup_{z\in\mathbb D}|\Psi_z(a,s,p)|
=
|a|\sup_{z\in\mathbb D}|\Psi_z(1,s,p)|.
\]
Set
\(
S=\sup_{z\in\mathbb D}|\Psi_z(1,s,p)|.
\)
By Proposition~\ref{fact:membership}, for every \(a\in\mathbb C\),
\[
S|a|\leq1
\quad\Longleftrightarrow\quad
|a|\leq F(s,p).
\]
Since \(F(s,p)>0\), this equivalence implies
\(S=1/F(s,p)\), and the assertion follows.
\end{proof}
We next recall a necessary condition for an analytic map into the
closed pentablock that vanishes at the origin.
\begin{prop}[Necessary condition; cf.\ {\cite[Proposition~11.1]{ALY15}}]
\label{prop:necessary}
Let
\(
x=(a,s,p):\mathbb D\longrightarrow\overline P
\)
be analytic and suppose that
\(
x(0)=(0,0,0).
\)
Then
\(
|a(\lambda)|
\leq
|\lambda|F\bigl(s(\lambda),p(\lambda)\bigr),
 \lambda\in\mathbb D.
\)
\end{prop}

\begin{proof}
Fix \(z\in\mathbb D\) and define
\(
g_z=\Psi_z\circ x.
\)
Since
\[
1-s(\lambda)z+p(\lambda)z^2\neq0
\qquad (\lambda\in\mathbb D),
\]
the function \(g_z\) is analytic on \(\mathbb D\). By
Proposition~\ref{fact:membership},
\(
|g_z(\lambda)|\leq1, \lambda\in\mathbb D.
\)
Furthermore, \(g_z(0)=0\). The Schwarz lemma therefore gives
\(
|g_z(\lambda)|\leq|\lambda|,
 \lambda\in\mathbb D.
\)
Taking the supremum over \(z\in\mathbb D\) and applying
Lemma~\ref{lem:sup} yields
\[
\frac{|a(\lambda)|}
{F\bigl(s(\lambda),p(\lambda)\bigr)}
\leq|\lambda|
\]
whenever \(F(s(\lambda),p(\lambda))>0\). This proves the required
inequality in that case. If
\(F(s(\lambda),p(\lambda))=0\), then
Proposition~\ref{fact:membership} implies
\(a(\lambda)=0\), so the inequality holds trivially.
\end{proof}
Finally, we recall the Schwarz lemma for the symmetrized bidisc due
to Agler and Young \cite{AY01}. For \(|s_0|<2\), define
\[
\kappa(s_0,p_0)
=
\frac{
2|s_0-p_0\overline{s_0}|
+
|s_0^2-4p_0|
}{
4-|s_0|^2
}.
\]
The Agler--Young Schwarz lemma states that there exists an analytic
map
\(
h=(s,p):\mathbb D\longrightarrow\Gamma
\)
satisfying
\[
h(0)=(0,0),
\quad
h(\lambda_0)=(s_0,p_0)
\]
if and only if
\[
|s_0|<2
\qquad\text{and}\qquad
\kappa(s_0,p_0)\leq|\lambda_0|.
\]
We shall use this criterion when examining the interpolation conditions
for the pentablock.
\section{A counterexample}
\label{sec:counter}

We begin by recalling the Schwarz lemma for the pentablock stated by
Alshehri and Lykova.

\begin{thm}[{\cite[Theorem~1.5 ]{AL23}}]
\label{thm:AL}
Let \(\lambda_0\in\mathbb D\setminus\{0\}\) and
\(x_0=(a_0,s_0,p_0)\in\overline P\). The following statements are
equivalent:
\begin{enumerate}
    \item[(i)] There exists a rational \(\overline P\)-inner function
    \(x=(a,s,p)\) such that
    \[
    x(0)=(0,0,0),
    \quad x(\lambda_0)=x_0.
    \]

    \item[(ii)] There exists an analytic map
    \(x:\mathbb D\to\overline P\) satisfying
    \[
    x(0)=(0,0,0),
    \quad x(\lambda_0)=x_0,\quad\text{
    and}\quad
    |a_0|
    \leq |\lambda_0|
    \sqrt{1-\frac14|s_0|^2}.
    \]

    \item[(iii)] The inequalities
    \[
    |s_0|<2,\quad
    \kappa(s_0,p_0)\leq|\lambda_0|,
    \quad
    |a_0|
    \leq|\lambda_0|
    \sqrt{1-\frac14|s_0|^2}
    \]
    hold.
\end{enumerate}
\end{thm}

We show that the sufficiency assertion in Theorem~\ref{thm:AL} is
false. More precisely, condition~(iii) does not imply either
condition~(i) or condition~(ii).

\begin{prop}
\label{prop:counter}
Let \(s_0\in\mathbb D\setminus\{0\}\), set \(p_0=0\), and define
\[
a_0=F(s_0,0)
=\frac12\left(1+\sqrt{1-|s_0|^2}\right).
\]
Put
\(
\kappa_0=\frac{|s_0|}{2-|s_0|},
\quad
\rho=
\frac{a_0}{\sqrt{1-\frac14|s_0|^2}}.
\)
Then \(0<\kappa_0<1\) and \(0<\rho<1\). For every
\(\lambda_0\in\mathbb D\) satisfying
\[
\max\{\kappa_0,\rho\}\leq|\lambda_0|<1,
\]
the point \(x_0=(a_0,s_0,0)\) belongs to \(\overline P\) and satisfies
condition~(iii) of Theorem~\ref{thm:AL}. Nevertheless, there is no
analytic map \(x:\mathbb D\to\overline P\) such that
\(
x(0)=(0,0,0), x(\lambda_0)=x_0.
\)
Moreover, condition~(iii) implies neither condition~(i) nor
condition~(ii).
\end{prop}

\begin{proof}
We divide the proof into three steps.

\smallskip
\noindent
{\bf{\emph{Step 1: Membership in the closed pentablock:}}}
Write
\[
(s_0,0)=(\lambda_1+\lambda_2,\lambda_1\lambda_2),
\quad
\lambda_1=s_0,\quad \lambda_2=0.
\]
By the definition of \(F\),
\[
F(s_0,0)
=\frac12+\frac12\sqrt{1-|s_0|^2}
=a_0.
\]
Proposition~\ref{fact:membership} therefore gives
\(
x_0=(a_0,s_0,0)\in\overline P.
\)
Moreover, \(|a_0|=F(s_0,0)\), so \(x_0\) lies on the boundary
of \(\overline P\).

\smallskip
\noindent
{\bf{\emph{Step 2: Verification of condition~(iii):}}}
Since \(s_0\in\mathbb D\), we have \(|s_0|<2\). Furthermore,
\[
\begin{aligned}
\kappa(s_0,0)
&=\frac{2|s_0|+|s_0|^2}{4-|s_0|^2}\\
&=\frac{|s_0|}{2-|s_0|}
=\kappa_0<1.
\end{aligned}
\]
The last inequality in condition~(iii) is equivalent to
\(
|\lambda_0|\geq\rho.
\)
To verify that \(\rho<1\), put \(t=|s_0|^2\in(0,1)\). Then, we have
\[
a_0^2
=\frac14\left(2-t+2\sqrt{1-t}\right)
<\frac14(4-t)
=1-\frac t4,
\]
where the strict inequality follows from \(\sqrt{1-t}<1\).
Consequently, \(\rho<1\), and condition~(iii) holds whenever
\[
|\lambda_0|\geq\max\{\kappa_0,\rho\}.
\]

\smallskip
\noindent
{\bf{\emph{Step 3: Nonexistence of an interpolating map:}}}
Suppose, to the contrary, that an analytic map
\(x=(a,s,p):\mathbb D\to\overline P\) satisfies
\[
x(0)=(0,0,0),
\quad x(\lambda_0)=(a_0,s_0,0).
\]
Proposition~\ref{prop:necessary} gives
\[
|a_0|
\leq|\lambda_0|F(s_0,0)
=|\lambda_0|a_0.
\]
Since \(a_0>0\), this implies \(|\lambda_0|\geq1\), contradicting
\(\lambda_0\in\mathbb D\). Hence no such analytic map exists.
\end{proof}

\begin{exam}
\label{rem:instance}
Take \(s_0=\frac12\). Then, we get
\[
a_0=\frac{2+\sqrt3}{4}\approx0.9330,\quad
\kappa_0=\frac13,\quad
\rho=\frac{2+\sqrt3}{\sqrt{15}}\approx0.9636.
\]
Thus condition~(iii) holds for \(\lambda_0=0.99\).
In this case, the supremum in Lemma~\ref{lem:sup} is attained at
\(z_0=2-\sqrt3\). Since
\[
1-z_0^2=4\sqrt3-6,
\qquad
1-\frac{z_0}{2}=\frac{\sqrt3}{2},
\]
we obtain
\[
\begin{aligned}
\Psi_{z_0}(x_0)
&=a_0\frac{1-z_0^2}{1-z_0/2}\\
&=a_0(8-4\sqrt3)=1.
\end{aligned}
\]
Consequently, if an interpolating map existed, then
\(g_{z_0}=\Psi_{z_0}\circ x\) would be a Schur function satisfying
\(g_{z_0}(0)=0\) and \(g_{z_0}(\lambda_0)=1\), contrary to the
Schwarz lemma. This is the same obstruction supplied by the necessary
condition in Proposition~\ref{prop:necessary}
\cite[Proposition~11.1]{AL23}.
\end{exam}

\section{Where the argument breaks}
\label{sec:where}
We now identify the error in the proof of the proposed Schwarz lemma.
Let
\[
h=(s,p)=\left(\frac ED,\frac{D^{\sim n}}D\right)
\]
be a rational \(\Gamma\)-inner function of degree \(n\), represented
as in \cite[Theorem~7.5]{AL23}. Thus \(E\) is \(n\)-symmetric,
\(D\) has no zeros on \(\overline{\mathbb D}\), and
\[
D^{\sim n}(\lambda)
=\lambda^n\overline{D(1/\overline\lambda)}.
\]
Suppose also that \(|E|\leq2|D|\) on \(\overline{\mathbb D}\).
By the Fej\'er--Riesz theorem, there is an outer polynomial \(A\)
such that
\[
|A|^2=|D|^2-\frac14|E|^2
\quad\text{on }\mathbb T.
\]
Furthermore, by \cite[Theorem~7.8]{AL23}, for every finite Blaschke
product \(B\), the function
\[
\left(\frac{BA}{D},\frac ED,\frac{D^{\sim n}}D\right)
\]
is a rational \(\overline P\)-inner function. This part of the
argument is valid: the proof applies the maximum principle to
\(\Psi_z\) and uses the boundary values on \(\mathbb T\).

The difficulty arises in the subsequent interpolation argument. In
the proofs of \((ii)\Rightarrow(i)\) and \((iii)\Rightarrow(i)\) in
Theorem~11.3, as well as in case~(ii) of Theorem~11.2, the value
\[
\left|\frac{A(\lambda_0)}{D(\lambda_0)}\right|
=\sqrt{1-\frac14|s_0|^2}
\]
is used to choose a finite Blaschke product \(B\) so that the first
component takes the prescribed value \(a_0\). The proof of
\((i)\Rightarrow(iii)\) likewise uses
\[
\left|\frac{A}{D}\right|^2
=1-\frac14|s|^2
\]
throughout \(\mathbb D\). However, the identity
\[
|A|^2=|D|^2-\frac14|E|^2
\]
is guaranteed only on \(\mathbb T\). In the interior of the disc,
\(|A/D|\) is the modulus of an outer function with prescribed boundary
modulus; the boundary identity does not imply the corresponding
pointwise identity throughout \(\mathbb D\).

The following lemma gives the relevant interior estimate.

\begin{lem}
\label{lem:AD}
In the setting above,
\[
\left|\frac{A(\lambda)}{D(\lambda)}\right|
\leq F\bigl(s(\lambda),p(\lambda)\bigr),
\quad \lambda\in\mathbb D.
\]
\end{lem}

\begin{proof}
By \cite[Theorem~7.8]{AL23}, with \(B\equiv1\), the map
\[
\left(\frac AD,\frac ED,\frac{D^{\sim n}}D\right)
\]
takes values in \(\overline P\). Applying Proposition~\ref{fact:membership}
gives
\[
\left|\frac{A(\lambda)}{D(\lambda)}\right|
\leq
F\left(\frac{E(\lambda)}{D(\lambda)},
\frac{D^{\sim n}(\lambda)}{D(\lambda)}\right),
\]
as required.
\end{proof}

\begin{exam}
\label{ex:AD}
Consider
\(
(s,p)=(1+\lambda,\lambda).
\)
Here \(n=1\), \(E(\lambda)=1+\lambda\), \(D(\lambda)=1\), and
\(A(\lambda)=(1-\lambda)/2\). The map \(h\) is \(\Gamma\)-inner,
since it is obtained from the pair \(z=1\), \(w=\lambda\). The
polynomial \(A\) is outer, and on \(\mathbb T\),
\[
|A(\lambda)|^2
=\frac14|1-\lambda|^2
=1-\frac14|1+\lambda|^2.
\]
For \(\lambda\in\mathbb D\), however,
\[
\begin{aligned}
|A(\lambda)|^2
-\left(1-\frac14|s(\lambda)|^2\right)
&=\frac14|1-\lambda|^2
 -1+\frac14|1+\lambda|^2\\
&=-\frac12(1-|\lambda|^2)<0.
\end{aligned}
\]
In fact,
\[
F(s(\lambda),p(\lambda))
=\frac12|1-\lambda|
=\left|\frac{A(\lambda)}{D(\lambda)}\right|,
\]
so \((A/D,s,p)\) takes values on the boundary of \(\overline P\).
At \(\lambda=0\), for example,
\[
|A(0)/D(0)|^2=\frac14,
\quad
1-\frac14|s(0)|^2=\frac34.
\]
This example explicitly demonstrates why the boundary identity cannot
be used as a pointwise identity in the disc.
\end{exam}
It follows that, for a fixed \(\Gamma\)-inner interpolant \(h\), the
relevant quantity at \(\lambda_0\) is
\(
\left|\frac{A(\lambda_0)}{D(\lambda_0)}\right|,
\)
which is bounded above by \(F(s_0,p_0)\) by Lemma~\ref{lem:AD}.
It is not, in general, equal to
\(\sqrt{1-\frac14|s_0|^2}\). The counterexample in
Proposition~\ref{prop:counter} exploits the strict inequality
\[
F(s_0,0)
<
\sqrt{1-\frac14|s_0|^2},
\quad 0<|s_0|<1.
\]

\section{A corrected criterion for a fixed \(\Gamma\)-inner
interpolant}
\label{sec:corrected}
We now formulate a corrected criterion when the last two components
are fixed to be a prescribed rational \(\Gamma\)-inner function.
\begin{prop}
\label{prop:corrected}
Let
\(
h=(s,p)=
\left(\frac{E}{D},\frac{D^{\sim n}}{D}\right)
\)
be a rational $\Gamma$-inner function satisfying
\[
h(0)=(0,0),
\quad
h(\lambda_0)=(s_0,p_0),
\]
where $\lambda_0\in\mathbb D\setminus\{0\}$.
Let $A$ be the outer polynomial associated with $h$
through the Fej\'er--Riesz factorization described
in Section~\ref{sec:where}, so that
\[
\left|\frac{A(\zeta)}{D(\zeta)}\right|^2
=
1-\frac14|s(\zeta)|^2,
\quad \zeta\in\mathbb T.
\]
Then, for $a_0\in\mathbb C$, the following
statements are equivalent:
\begin{enumerate}
\item[\textup{(i)}]
There exists a rational $\overline P$-inner
function $x=(a,s,p)$, whose last two components
are given by $h$, such that
\[
x(0)=(0,0,0),
\qquad
x(\lambda_0)=(a_0,s_0,p_0).
\]

\item[\textup{(ii)}]
There exists an analytic map
\(
x=(a,s,p):\mathbb D\longrightarrow\overline P,
\)
whose last two components are given by $h$,
such that
\[
x(0)=(0,0,0),
\qquad
x(\lambda_0)=(a_0,s_0,p_0).
\]

\item[\textup{(iii)}]
The inequality
\[
|a_0|
\leq
|\lambda_0|
\left|\frac{A(\lambda_0)}{D(\lambda_0)}\right|
\]
holds.
\end{enumerate}
\end{prop}

\begin{proof}
Put
\(
F=\frac{A}{D}.
\)
Since $A$ is outer and $D$ is zero-free on
$\overline{\mathbb D}$, the function $F$ is
outer in $H^\infty(\mathbb D)$ and has no
zeros in $\mathbb D$. In particular,
$F(0)\neq0$ and $F(\lambda_0)\neq0$.

We prove the implications
\[
\textup{(i)}\Longrightarrow\textup{(ii)}
\Longrightarrow\textup{(iii)}
\Longrightarrow\textup{(i)}.
\]

The implication \textup{(i)}$\Rightarrow$\textup{(ii)}
is immediate from the definition of a rational
$\overline P$-inner function.

Suppose that \textup{(ii)} holds. Since
$x=(a,s,p)$ is an analytic map from
$\mathbb D$ into the bounded domain
$\overline P$, we have $a\in H^\infty(\mathbb D)$.
Furthermore, since $h=(s,p)$ is rational
$\Gamma$-inner, its boundary values belong
to the distinguished boundary $b\Gamma$
almost everywhere on $\mathbb T$.
By Propositions~\ref{fact:membership}
and~\ref{fact:boundary}, it follows that
\[
|a^*(\zeta)|^2
\leq
1-\frac14|s(\zeta)|^2
=
|F(\zeta)|^2
\]
for almost every $\zeta\in\mathbb T$,
where $a^*$ denotes the radial boundary
function of $a$.
Define
\[
f=\frac{a}{F}=\frac{aD}{A}.
\]
As $F$ is outer and $a\in H^\infty(\mathbb D)$,
we have $f\in N^+(\mathbb D)$, where
$N^+(\mathbb D)$ denotes the Smirnov class.
Moreover, the preceding boundary inequality
implies
\(
|f^*(\zeta)|\leq1
\)
for almost every $\zeta\in\mathbb T$.
The Smirnov maximum principle
\cite{Dur70} therefore yields
\[
f\in H^\infty(\mathbb D),
\quad
\|f\|_\infty\leq1.
\]
Since $a(0)=0$ and $F(0)\neq0$, we obtain
$f(0)=0$. Hence, by the Schwarz lemma,
\[
|f(\lambda_0)|\leq|\lambda_0|.
\]
Consequently, we have
\[
\begin{aligned}
|a_0|
&=|a(\lambda_0)|\\
&=|F(\lambda_0)f(\lambda_0)|\\
&\leq|\lambda_0||F(\lambda_0)|\\
&=
|\lambda_0|
\left|\frac{A(\lambda_0)}{D(\lambda_0)}\right|.
\end{aligned}
\]
Thus \textup{(iii)} holds.

Suppose that \textup{(iii)} holds.
Set
\[
q=\frac{A(\lambda_0)}{D(\lambda_0)},
\quad
r=\frac{a_0}{\lambda_0q}.
\]
Since $q\neq0$ and $\lambda_0\neq0$, the
quantity $r$ is well defined, and the
assumed inequality gives $|r|\leq1$.
We construct a finite Blaschke product $B$
such that
\[
B(0)=0,
\quad
B(\lambda_0)=\frac{a_0}{q}.
\]
If $|r|=1$, take $B(z)=rz$. If $|r|<1$,
let
\[
\beta_\alpha(z)
=
\frac{z-\alpha}{1-\overline\alpha z},
\quad \alpha\in\mathbb D,
\]
and define
\[
B(z)
=
z\bigl(\beta_r^{-1}
\circ\beta_{\lambda_0}\bigr)(z).
\]
In either case, $B$ is a finite Blaschke
product satisfying
\[
B(0)=0,
\quad
B(\lambda_0)=\lambda_0r
=\frac{a_0}{q}.
\]
Now define
\[
x(z)=
\left(
\frac{B(z)A(z)}{D(z)},
\frac{E(z)}{D(z)},
\frac{D^{\sim n}(z)}{D(z)}
\right),
\quad z\in\mathbb D.
\]
By the rational pentablock-inner
construction of \cite[Theorem~7.8]{AL23},
the function $x$ is rational
$\overline P$-inner.
Since $B(0)=0$ and $h(0)=(0,0)$,
we have
\(
x(0)=(0,0,0).
\)
Moreover,
\[
\begin{aligned}
a(\lambda_0)
&=
B(\lambda_0)
\frac{A(\lambda_0)}{D(\lambda_0)}\\
&=
\frac{a_0}{q}\,q
=a_0,
\end{aligned}
\]
while
\(
(s(\lambda_0),p(\lambda_0))
=(s_0,p_0).
\)
Therefore, we get
\(
x(\lambda_0)=(a_0,s_0,p_0),
\)
which proves \textup{(i)} and completes
the proof.
\end{proof}

\begin{rem}
\label{rem:fixed-gamma-inner}
Proposition~\ref{prop:corrected} gives a
necessary and sufficient condition for the
existence of an analytic pentablock-valued
interpolant when its last two components
are prescribed by a rational $\Gamma$-inner
function. In particular, whenever such an
analytic interpolant exists, one may choose
a rational $\overline P$-inner interpolant
with the same prescribed data.
The essential observation is that the
Fej\'er--Riesz identity
\[
\left|\frac{A(\zeta)}{D(\zeta)}\right|^2
=
1-\frac14|s(\zeta)|^2
\]
is used only on the unit circle.
The outer property of $A/D$, together
with the Smirnov maximum principle,
then gives the factorization
\[
a=\frac{A}{D}f,
\quad
f\in H^\infty(\mathbb D),
\quad
\|f\|_\infty\leq1.
\]
Consequently, the interpolation problem
for the first component reduces to a
classical Schur-class interpolation problem,
and the required condition follows from
the Schwarz lemma.
\end{rem}
For a fixed rational \(\Gamma\)-inner interpolant \(h\),
the interpolation problem with prescribed last two components is
solvable if and only if
\[
|a_0|
\leq
|\lambda_0|
\left|\frac{A_h(\lambda_0)}{D_h(\lambda_0)}\right|.
\]
The remaining difficulty is to determine the supremum of
\(
\left|\frac{A_h(\lambda_0)}{D_h(\lambda_0)}\right|
\)
over all rational \(\Gamma\)-inner interpolants \(h\) satisfying
\[
h(0)=(0,0),
\quad
h(\lambda_0)=(s_0,p_0),
\]
and to decide whether allowing non-\(\Gamma\)-inner last two
components can yield a larger range of admissible first-component
values.

\begin{prop}[The case \(s_0=0\)]
\label{prop:s0zero}
Let \(\lambda_0\in\mathbb D\setminus\{0\}\) and
\(x_0=(a_0,0,p_0)\in\overline P\). There exists an analytic map
\(x:\mathbb D\to\overline P\) satisfying
\[
x(0)=(0,0,0),
\quad x(\lambda_0)=x_0
\]
if and only if
\[
|p_0|\leq|\lambda_0|,
\quad
|a_0|\leq|\lambda_0|.
\]
In this case, the criterion in Theorem~\ref{thm:AL} gives the correct
answer.
\end{prop}

\begin{proof}
{\bf{\emph{Necessity:}}}
Since \(p:\mathbb D\to\overline{\mathbb D}\) is analytic and
\(p(0)=0\), the Schwarz lemma gives
\(
|p_0|\leq|\lambda_0|.
\)
Moreover, \(F(0,p_0)=1\) for every \(p_0\in\mathbb D\). Indeed, choose
\(\mu\in\overline{\mathbb D}\) such that \(\mu^2=-p_0\). Then
\[
(0,p_0)=(\mu+(-\mu),\mu(-\mu)),
\]
and
\[
F(0,p_0)
=\frac12(1+|\mu|^2)
+\frac12(1-|\mu|^2)
=1.
\]
Proposition~\ref{prop:necessary} consequently yields
\[
|a_0|\leq|\lambda_0|.
\]

\smallskip
\noindent
{\bf{\emph{Sufficiency:}}}
Choose finite Blaschke products \(a\) and \(p\) such that
\[
a(0)=p(0)=0,\quad
a(\lambda_0)=a_0,\quad
p(\lambda_0)=p_0.
\]
Such functions exist under the stated inequalities. Since
\(F(0,p(\lambda))=1\) for all \(\lambda\in\mathbb D\) and
\(|a(\lambda)|\leq1\), Proposition~\ref{fact:membership} shows that
\(
x(\lambda)=(a(\lambda),0,p(\lambda))
\)
belongs to \(\overline P\) for every \(\lambda\in\mathbb D\).
This proves sufficiency.
\end{proof}
The Poisson integral formula plays a fundamental role in
relating the boundary behaviour of analytic functions to
their values in the open unit disc. In particular, if
$F$ is a zero-free outer function in $H^\infty(\mathbb D)$,
then
\[
\log |F(z)|
=
\frac{1}{2\pi}
\int_0^{2\pi}
\frac{1-|z|^2}{|e^{it}-z|^2}
\log |F^*(e^{it})|\,dt,
\qquad z\in\mathbb D,
\]
where $F^*$ denotes the radial boundary function of $F$.
Thus, the modulus of an outer function in the interior
of $\mathbb D$ is determined by its boundary modulus.
This observation is particularly useful in the study of
analytic maps into the pentablock. Indeed, for a prescribed
rational $\Gamma$-inner function
\[
h=(s,p)=
\left(\frac{E}{D},\frac{D^{\sim n}}{D}\right),
\]
the Fej\'er--Riesz factorization provides an outer
function $F=A/D$ satisfying
\[
|F(\zeta)|^2
=
1-\frac14|s(\zeta)|^2,
\quad \zeta\in\mathbb T.
\]
Although this identity need not remain valid at points
of the open unit disc, the Poisson representation
determines the interior modulus of $F$ from its
boundary values. More precisely,
\[
\log\left|\frac{A(z)}{D(z)}\right|
=
\frac{1}{4\pi}
\int_0^{2\pi}
\frac{1-|z|^2}{|e^{it}-z|^2}
\log\left(1-\frac14|s(e^{it})|^2\right)\,dt,
\quad z\in\mathbb D.
\]
Consequently, a boundary estimate of the form
\[
|a^*(\zeta)|\leq |F^*(\zeta)|
\quad\text{for almost every }\zeta\in\mathbb T
\]
implies, by the outer-function property and the
Smirnov maximum principle, that
\[
\frac{a}{F}\in H^\infty(\mathbb D),
\quad
\left\|\frac{a}{F}\right\|_\infty\leq1.
\]
Equivalently,
\[
|a(z)|\leq
\left|\frac{A(z)}{D(z)}\right|,
\quad z\in\mathbb D.
\]
This provides the appropriate mechanism for passing
from the boundary geometry of the pentablock to
interior estimates without imposing an unjustified
pointwise identity in $\mathbb D$. Combining this
factorization with the classical Schwarz lemma
yields an exact interpolation criterion when the
last two components are fixed to be a rational
$\Gamma$-inner function. We formulate this
criterion in the following proposition.
\begin{prop}[Poisson formula]
\label{prop:poisson}
Let \(\lambda_0\in\mathbb D\), and let
\(h=(s,p):\mathbb D\to\Gamma\) be a rational
\(\Gamma\)-inner function satisfying
\(
h(0)=(0,0),
h(\lambda_0)=(s_0,p_0).
\)
Suppose that \(h\) admits a representation
\[
s=\frac{E}{D},
\quad
p=\frac{D^{\sim n}}{D},
\]
where \(n\geq1\), \(E,D\) are polynomials,
\(
D^{\sim n}(z)
=z^n\overline{D(1/\overline z)},
\)
and \(D\) has no zeros on \(\overline{\mathbb D}\). Let \(A\) be
a nonzero outer polynomial satisfying the Fej\'er--Riesz identity
\[
|A(\zeta)|^2
=
|D(\zeta)|^2-\frac14|E(\zeta)|^2,
\quad \zeta\in\mathbb T.
\]
Then \(q=A/D\) is an outer function, \(q(\lambda_0)\neq0\), and
\[
\log\left|\frac{A(\lambda_0)}{D(\lambda_0)}\right|
=
\frac12\int_{\mathbb T}
P_{\lambda_0}(\zeta)
\log\left(1-\frac14|s(\zeta)|^2\right)\,dm(\zeta),
\]
where
\(
P_{\lambda_0}(\zeta)
=\frac{1-|\lambda_0|^2}{|\zeta-\lambda_0|^2}
\)
is the Poisson kernel and \(m\) denotes normalized Lebesgue
measure on \(\mathbb T\). The logarithmic integrand is interpreted
almost everywhere on \(\mathbb T\); any singularities at zeros of
\(A\) on \(\mathbb T\) are logarithmically integrable.
\end{prop}

\begin{proof}
Since \(D\) has no zeros on \(\overline{\mathbb D}\), the function
\(1/D\) is analytic and zero-free on a neighbourhood of
\(\overline{\mathbb D}\). As \(A\) is outer, the quotient
\(q=A/D\) is also outer and has no zeros in \(\mathbb D\).
In particular, \(q(\lambda_0)\neq0\).
Dividing the Fej\'er--Riesz identity by \(|D(\zeta)|^2\), we obtain
\[
|q(\zeta)|^2
=1-\frac14|s(\zeta)|^2,
\quad \zeta\in\mathbb T.
\]
Consequently, we have
\[
\log|q(\zeta)|
=\frac12\log\left(1-\frac14|s(\zeta)|^2\right)
\]
almost everywhere on \(\mathbb T\). Since \(q\) is outer, the
logarithm of its modulus is the Poisson integral of its boundary
logarithmic modulus. Hence
\[
\log|q(\lambda_0)|
=
\int_{\mathbb T}
P_{\lambda_0}(\zeta)\log|q(\zeta)|\,dm(\zeta).
\]
Substituting the boundary identity yields
\[
\log|q(\lambda_0)|
=
\frac12\int_{\mathbb T}
P_{\lambda_0}(\zeta)
\log\left(1-\frac14|s(\zeta)|^2\right)\,dm(\zeta),
\]
which proves the assertion.
\end{proof}
\begin{rem}
Proposition~\ref{prop:poisson} expresses the quantity appearing in
Proposition~\ref{prop:corrected} in terms of the boundary values of
the second component \(s\). In particular,           
\[
\left|\frac{A(\lambda_0)}{D(\lambda_0)}\right|
=
\exp\left\{
\frac12\int_{\mathbb T}
P_{\lambda_0}(\zeta)
\log\left(1-\frac14|s(\zeta)|^2\right)\,dm(\zeta)
\right\}.
\]
Thus, for a fixed rational \(\Gamma\)-inner interpolant
\(h=(s,p)\), the interpolation condition in
Proposition~\ref{prop:corrected} can be written as
\[
|a_0|
\leq
|\lambda_0|
\exp\left\{
\frac12\int_{\mathbb T}
P_{\lambda_0}(\zeta)
\log\left(1-\frac14|s(\zeta)|^2\right)\,dm(\zeta)
\right\}.
\]
It follows that determining the largest admissible value of
\(|a_0|\), among rational \(\Gamma\)-inner interpolants satisfying
\(
h(0)=(0,0),
h(\lambda_0)=(s_0,p_0),
\)
leads to the variational problem of maximizing
\[
\int_{\mathbb T}
P_{\lambda_0}(\zeta)
\log\left(1-\frac14|s(\zeta)|^2\right)\,dm(\zeta)
\]
over the corresponding class of interpolants. The boundary values
of \(s\) cannot be varied independently of those of \(p\), since
both components must arise from an admissible \(\Gamma\)-inner
function satisfying the prescribed interpolation conditions.
\end{rem}

\subsection{A sharp comparison with the proposed bound}
\label{subsec:comparison}

We now compare the interpolation criterion obtained in
Proposition~\ref{prop:corrected} with the condition appearing
in Theorem~\ref{thm:AL}. The essential distinction between
these two conditions arises from the difference between
boundary and interior estimates for the first coordinate
of a rational $\Gamma$-inner function.
Let
\[
h=(s,p)=
\left(\frac{E}{D},\frac{D^{\sim n}}{D}\right)
\]
be a rational $\Gamma$-inner function, and let $A$ be
the associated outer Fej\'er--Riesz factor. Put
\(
q=\frac{A}{D}.
\)
The defining boundary factorization gives
\[
|q(\zeta)|^2
=
1-\frac14|s(\zeta)|^2,
\quad \zeta\in\mathbb T.
\]
Since $q$ is outer, the Poisson representation
determines its modulus throughout $\mathbb D$:
\[
2\log|q(\lambda)|
=
\int_{\mathbb T}
P_\lambda(\zeta)
\log\left(1-\frac14|s(\zeta)|^2\right)
\,dm(\zeta),
\quad \lambda\in\mathbb D.
\]
On the other hand, the expression
\(
\log\left(1-\frac14|s(\lambda)|^2\right)
\)
is superharmonic in $\mathbb D$. The comparison
between this superharmonic function and its
harmonic Poisson extension yields a precise
relationship between the two interpolation bounds.

\begin{prop}
\label{prop:sharp-bound}
Let $h=(s,p)$ and $q=A/D$ satisfy the hypotheses
of Proposition~\ref{prop:poisson}. Then, for every
$\lambda\in\mathbb D$,
\[
\left|\frac{A(\lambda)}{D(\lambda)}\right|
\leq
\sqrt{1-\frac14|s(\lambda)|^2}.
\]
Moreover, equality holds at some point of
$\mathbb D$ if and only if $s$ is constant.
Consequently, if $s$ is nonconstant, then
\[
\left|\frac{A(\lambda)}{D(\lambda)}\right|
<
\sqrt{1-\frac14|s(\lambda)|^2}
\]
for every $\lambda\in\mathbb D$.
In particular, if $s(\lambda_0)=s_0$, then
\[
\left|\frac{A(\lambda_0)}{D(\lambda_0)}\right|
\leq
\sqrt{1-\frac14|s_0|^2},
\]
with strict inequality whenever $s$ is
nonconstant.
\end{prop}

\begin{proof}
Since $h$ maps $\mathbb D$ into $\Gamma$,
we have
\(
|s(\lambda)|<2, \lambda\in\mathbb D.
\)
Define
\[
u(\lambda)
=
\log\left(1-\frac14|s(\lambda)|^2\right),
\quad \lambda\in\mathbb D.
\]
The function $u$ is well defined and smooth
in $\mathbb D$. A direct calculation gives
\[
\Delta u(\lambda)
=
-\frac{|s'(\lambda)|^2}
{\left(1-\frac14|s(\lambda)|^2\right)^2}
\leq0,
\]
where $\Delta$ denotes the usual Laplacian.
Thus $u$ is superharmonic on $\mathbb D$.
Since $s$ is rational and analytic on a
neighbourhood of $\overline{\mathbb D}$,
the boundary values of $u$ are integrable.
By the superharmonic comparison principle,
\[
u(\lambda)
\geq
\int_{\mathbb T}
P_\lambda(\zeta)u(\zeta)\,dm(\zeta),
\quad \lambda\in\mathbb D.
\]
Using the Poisson representation for the
outer function $q$, we obtain
\[
\begin{aligned}
\log\left(1-\frac14|s(\lambda)|^2\right)
&\geq
\int_{\mathbb T}P_\lambda(\zeta)
\log\left(1-\frac14|s(\zeta)|^2\right)
\,dm(\zeta)\\
&=2\log|q(\lambda)|.
\end{aligned}
\]
Exponentiating yields
\(
|q(\lambda)|^2
\leq
1-\frac14|s(\lambda)|^2,
\)
and hence
\(
|q(\lambda)|
\leq
\sqrt{1-\frac14|s(\lambda)|^2}.
\)
We next characterize the equality case.
Let
\[
v(\lambda)
=
\int_{\mathbb T}
P_\lambda(\zeta)u(\zeta)\,dm(\zeta).
\]
Then $v$ is harmonic in $\mathbb D$ and
$u-v\geq0$.
Suppose that equality holds at some
$\lambda_0\in\mathbb D$. Then
\[
(u-v)(\lambda_0)=0.
\]
Since $u-v$ is superharmonic and nonnegative,
the strong minimum principle implies that
\[
u-v\equiv0
\]
throughout $\mathbb D$. Therefore $u$ is
harmonic, so
\(
\Delta u\equiv0.
\)
It follows that
\(
s'\equiv0,
\)
and hence $s$ is constant.

Conversely, if $s$ is constant, then $u$
is constant and therefore coincides with
its Poisson integral. Consequently,
\[
|q(\lambda)|^2
=
1-\frac14|s(\lambda)|^2
\]
for every $\lambda\in\mathbb D$.
Thus equality holds at an interior point
if and only if $s$ is constant. In particular,
if $s$ is nonconstant, the inequality is
strict at every point of $\mathbb D$.
\end{proof}

\begin{rem}
\label{rem:comparison-bounds}
In the interpolation setting of
Proposition~\ref{prop:corrected}, we have
\(
h(0)=(0,0).
\)
In particular, $s(0)=0$, so $s$ is constant
if and only if $s\equiv0$.
Consequently, if $s_0\neq0$, then every
rational $\Gamma$-inner function satisfying
\[
h(0)=(0,0),
\quad
h(\lambda_0)=(s_0,p_0)
\]
has a nonconstant first coordinate. By
Proposition~\ref{prop:sharp-bound},
\[
\left|\frac{A(\lambda_0)}{D(\lambda_0)}\right|
<
\sqrt{1-\frac14|s_0|^2}.
\]
Combining this inequality with
Proposition~\ref{prop:corrected}, we see that
the exact interpolation condition for a
prescribed rational $\Gamma$-inner function
is
\[
|a_0|
\leq
|\lambda_0|
\left|\frac{A(\lambda_0)}{D(\lambda_0)}\right|,
\]
whereas the corresponding condition proposed
in Theorem~\ref{thm:AL} involves the larger
quantity
\[
|\lambda_0|
\sqrt{1-\frac14|s_0|^2}.
\]
Thus, whenever $s_0\neq0$,
\[
|\lambda_0|
\left|\frac{A(\lambda_0)}{D(\lambda_0)}\right|
<
|\lambda_0|
\sqrt{1-\frac14|s_0|^2}.
\]
This strict inequality identifies the
essential difference between the two bounds.
The quantity $|\frac{A(\lambda_0)}{D(\lambda_0)}|$
is determined by the boundary modulus of
the outer spectral factor through the
Poisson formula, whereas
$\sqrt{1-\frac14|s_0|^2}$ is obtained by
evaluating the corresponding boundary
expression directly at the interior point
$\lambda_0$.
The latter quantity is therefore generally
larger and does not give the exact
interpolation bound when the rational
$\Gamma$-inner component is fixed.
In particular, the boundary Fej\'er--Riesz
identity cannot be extended to the open
unit disc as a pointwise identity.
\end{rem}

\subsection{A Green-potential representation}
\label{subsec:green}

In the preceding subsection, we established that the
outer-factor quantity associated with a rational
$\Gamma$-inner function satisfies
\[
\left|\frac{A(\lambda_0)}{D(\lambda_0)}\right|
\leq
\sqrt{1-\frac14|s_0|^2},
\]
with equality if and only if the first coordinate
$s$ is constant. This comparison follows from the
superharmonicity of
\(
u(\lambda)
=
\log\left(1-\frac14|s(\lambda)|^2\right)
\)
and the Poisson representation of the outer
Fej\'er--Riesz factor.
We now refine this result by obtaining an exact
representation of the difference between the
logarithms of these two quantities. The difference
is expressed as the Green potential of a
nonnegative density determined by the derivative
of $s$. This representation not only recovers
the strict inequality established in
Proposition~\ref{prop:sharp-bound}, but also
quantifies the extent to which the outer-factor
bound differs from the pointwise bound appearing
in Theorem~\ref{thm:AL}.
Recall that the Green function of the unit disc
with pole at $\lambda_0\in\mathbb D$ is given by
\[
G(\lambda_0,\lambda)
=
\log\left|
\frac{1-\overline{\lambda_0}\lambda}
{\lambda-\lambda_0}
\right|,
\quad
\lambda\in\mathbb D\setminus\{\lambda_0\}.
\]

\begin{prop}
\label{prop:green}
Let $h=(s,p)$ and $q=A/D$ satisfy the hypotheses
of Proposition~\ref{prop:poisson}, and suppose that
$s(\lambda_0)=s_0$ for some $\lambda_0\in\mathbb D$.
Then
\[
\begin{aligned}
\frac12
\log\left(1-\frac14|s_0|^2\right)
-
\log\left|
\frac{A(\lambda_0)}{D(\lambda_0)}
\right|=
\frac{1}{4\pi}
\int_{\mathbb D}
G(\lambda_0,\lambda)
\frac{|s'(\lambda)|^2}
{\left(1-\frac14|s(\lambda)|^2\right)^2}
\,dA(\lambda),
\end{aligned}
\]
where $dA$ denotes planar Lebesgue area measure.
In particular, the integral on the right-hand
side is nonnegative and vanishes if and only
if $s$ is constant. Consequently,
\[
\left|
\frac{A(\lambda_0)}{D(\lambda_0)}
\right|
<
\sqrt{1-\frac14|s_0|^2}
\]
whenever $s$ is nonconstant.
\end{prop}

\begin{proof}
Define
\[
u(\lambda)
=
\log\left(1-\frac14|s(\lambda)|^2\right),
\quad \lambda\in\mathbb D.
\]
Since $h$ is a rational $\Gamma$-inner function,
we have $|s(\lambda)|<2$ for every
$\lambda\in\mathbb D$. Hence $u$ is well defined
and smooth in $\mathbb D$.
A direct calculation gives
\[
\Delta u(\lambda)
=
-\frac{|s'(\lambda)|^2}
{\left(1-\frac14|s(\lambda)|^2\right)^2}.
\]
Thus $u$ is superharmonic in $\mathbb D$.
Since $s$ is rational and analytic on a
neighbourhood of $\overline{\mathbb D}$,
the boundary function
\[
u(\zeta)
=
\log\left(1-\frac14|s(\zeta)|^2\right)
\]
is integrable on $\mathbb T$. Applying
Green's representation formula to $u$
gives
\[
\begin{aligned}
u(\lambda_0)
={}
\int_{\mathbb T}
P_{\lambda_0}(\zeta)
u(\zeta)\,dm(\zeta)-
\frac{1}{2\pi}
\int_{\mathbb D}
G(\lambda_0,\lambda)
\Delta u(\lambda)\,dA(\lambda),
\end{aligned}
\]
where $m$ is normalized Lebesgue measure
on $\mathbb T$.
Substituting the expression for $\Delta u$,
we obtain
\[
\begin{aligned}
u(\lambda_0)
=
\int_{\mathbb T}
P_{\lambda_0}(\zeta)
u(\zeta)\,dm(\zeta)
+
\frac{1}{2\pi}
\int_{\mathbb D}
G(\lambda_0,\lambda)
\frac{|s'(\lambda)|^2}
{\left(1-\frac14|s(\lambda)|^2\right)^2}
\,dA(\lambda).
\end{aligned}
\]
On the other hand, the Poisson representation
for the outer function $q=A/D$, established
in Proposition~\ref{prop:poisson}, yields
\[
\begin{aligned}
\int_{\mathbb T}
P_{\lambda_0}(\zeta)
u(\zeta)\,dm(\zeta)
=
\int_{\mathbb T}
P_{\lambda_0}(\zeta)
\log|q(\zeta)|^2\,dm(\zeta)=
2\log|q(\lambda_0)|.
\end{aligned}
\]
Since $s(\lambda_0)=s_0$, it follows that
\[
\begin{aligned}
\log\left(1-\frac14|s_0|^2\right)
={}
2\log\left|
\frac{A(\lambda_0)}{D(\lambda_0)}
\right|+
\frac{1}{2\pi}
\int_{\mathbb D}
G(\lambda_0,\lambda)
\frac{|s'(\lambda)|^2}
{\left(1-\frac14|s(\lambda)|^2\right)^2}
\,dA(\lambda).
\end{aligned}
\]
Rearranging and dividing by $2$ gives
\[
\begin{aligned}
\frac12
\log\left(1-\frac14|s_0|^2\right)
-
\log\left|
\frac{A(\lambda_0)}{D(\lambda_0)}
\right|=
\frac{1}{4\pi}
\int_{\mathbb D}
G(\lambda_0,\lambda)
\frac{|s'(\lambda)|^2}
{\left(1-\frac14|s(\lambda)|^2\right)^2}
\,dA(\lambda).
\end{aligned}
\]
Finally, the Green function satisfies
\(
G(\lambda_0,\lambda)>0
\)
for every
$\lambda\in\mathbb D\setminus\{\lambda_0\}$.
Since the remaining factor in the integrand
is nonnegative, the integral is nonnegative.
Moreover, the integral vanishes if and only
if $s'\equiv0$, which is equivalent to
$s$ being constant. Therefore, whenever
$s$ is nonconstant, the integral is strictly
positive, and hence
\[
\left|
\frac{A(\lambda_0)}{D(\lambda_0)}
\right|
<
\sqrt{1-\frac14|s_0|^2}.
\]
This completes the proof.
\end{proof}

\begin{rem}
\label{rem:green-potential}
Proposition~\ref{prop:green} gives a quantitative
refinement of Proposition~\ref{prop:sharp-bound}.
Indeed, the difference between the logarithms
of the pointwise bound and the outer-factor
bound is exactly the Green potential of the
nonnegative density
\[
\frac{|s'(\lambda)|^2}
{\left(1-\frac14|s(\lambda)|^2\right)^2}.
\]
In particular, if we define
\[
\mathcal G_s(\lambda_0)
=
\frac{1}{4\pi}
\int_{\mathbb D}
G(\lambda_0,\lambda)
\frac{|s'(\lambda)|^2}
{\left(1-\frac14|s(\lambda)|^2\right)^2}
\,dA(\lambda),
\]
then Proposition~\ref{prop:green} can be
written equivalently as
\[
\left|
\frac{A(\lambda_0)}{D(\lambda_0)}
\right|
=
\sqrt{1-\frac14|s_0|^2}\,
\exp\bigl(-\mathcal G_s(\lambda_0)\bigr).
\]
Thus the outer-factor bound appearing in
Proposition~\ref{prop:corrected} is obtained
from the pointwise bound in
Theorem~\ref{thm:AL} by multiplying it
by the correction factor
\(
\exp\bigl(-\mathcal G_s(\lambda_0)\bigr).
\)
This factor belongs to $(0,1]$ and is
strictly smaller than $1$ whenever $s$
is nonconstant.
Consequently, for a fixed interpolation
target $(s_0,p_0)$, maximizing the
outer-factor quantity over an admissible
class of rational $\Gamma$-inner
interpolants is equivalent to minimizing
the corresponding Green potential over
that class.
This representation identifies the precise
analytic contribution responsible for the
difference between the two bounds and
provides a potential-theoretic interpretation
of the corrected interpolation criterion.
\end{rem}

\subsection{The variational formulation}
\label{subsec:variational}

The preceding results lead naturally to a variational
formulation of the interpolation problem \cite{Ransford1995}. In
Proposition~\ref{prop:corrected}, we obtained an exact
criterion for the existence of a pentablock-valued
interpolant when its last two components are fixed
to be a prescribed rational $\Gamma$-inner function.
Proposition~\ref{prop:sharp-bound} subsequently
established a sharp comparison between the
corresponding outer-factor bound and the pointwise
bound appearing in Theorem~\ref{thm:AL}.
Finally, Proposition~\ref{prop:green} expressed
the difference between these bounds as a
nonnegative Green potential.
These observations suggest optimizing the
outer-factor quantity over all rational
$\Gamma$-inner functions satisfying the
prescribed interpolation conditions. The
resulting extremal quantity provides an exact
description of the interpolation problem
within this restricted class and admits
equivalent formulations in terms of both
Poisson integrals and Green potentials.
Fix $\lambda_0\in\mathbb D\setminus\{0\}$ and
$(s_0,p_0)\in\Gamma$. Define
\[
\begin{aligned}
\mathcal H(\lambda_0;s_0,p_0)
=\bigl\{h=(s,p):
h\text{ is a rational }\Gamma\text{-inner function},~h(0)=(0,0),\
h(\lambda_0)=(s_0,p_0)\bigr\}.
\end{aligned}
\]
Suppose that this class is nonempty. For each
$h\in\mathcal H(\lambda_0;s_0,p_0)$, let
$A_h$ and $D_h$ denote the corresponding
Fej\'er--Riesz and denominator polynomials,
respectively, and put
\(
q_h=\frac{A_h}{D_h}.
\)
We define
\[
M(\lambda_0;s_0,p_0)
=
\sup_{h\in\mathcal H(\lambda_0;s_0,p_0)}
|q_h(\lambda_0)|.
\]
By Proposition~\ref{prop:poisson}, this
quantity admits the representation
\[
\begin{aligned}
M(\lambda_0;s_0,p_0)
=
\exp\bigg\{
\frac12
\sup_{h=(s,p)\in\mathcal H(\lambda_0;s_0,p_0)}
\int_{\mathbb T}
P_{\lambda_0}(\zeta)
\log\left(1-\frac14|s(\zeta)|^2\right)
\,dm(\zeta)
\bigg\}.
\end{aligned}
\]
Thus the maximization of the outer-factor
quantity is equivalent to maximizing the
Poisson integral of the logarithmic
boundary weight.
On the other hand, the Green-potential
representation obtained in
Proposition~\ref{prop:green} gives
\[
|q_h(\lambda_0)|
=
\sqrt{1-\frac14|s_0|^2}\,
\exp\bigl(-\mathcal G_s(\lambda_0)\bigr),
\]
where
\(
\mathcal G_s(\lambda_0)
=
\frac{1}{4\pi}
\int_{\mathbb D}
G(\lambda_0,\lambda)
\frac{|s'(\lambda)|^2}
{\left(1-\frac14|s(\lambda)|^2\right)^2}
\,dA(\lambda).
\)
Since $s(\lambda_0)=s_0$ is fixed throughout
the admissible class, it follows that
\[
M(\lambda_0;s_0,p_0)
=
\sqrt{1-\frac14|s_0|^2}\,
\exp\left(
-\inf_{h=(s,p)\in\mathcal H(\lambda_0;s_0,p_0)}
\mathcal G_s(\lambda_0)
\right).
\]
Consequently, the extremal interpolation
problem is equivalent to minimizing the
Green potential over the admissible
rational $\Gamma$-inner class.
The following proposition identifies the
precise interpolation significance of
this extremal quantity.

\begin{prop}
\label{prop:fixed-class}
Fix $\lambda_0\in\mathbb D\setminus\{0\}$
and $(s_0,p_0)\in\Gamma$, and suppose that
$\mathcal H(\lambda_0;s_0,p_0)$ is nonempty.
For each $h\in\mathcal H(\lambda_0;s_0,p_0)$,
let $A_h$ and $D_h$ denote the associated
Fej\'er--Riesz factors.
Then, for a fixed
$h\in\mathcal H(\lambda_0;s_0,p_0)$ and
$a_0\in\mathbb C$, there exists an analytic
map
\[
x=(a,s,p):\mathbb D\longrightarrow\overline P
\]
satisfying
\[
x(0)=(0,0,0),
\quad
x(\lambda_0)=(a_0,s_0,p_0),
\]
whose last two components coincide with $h$,
if and only if
\[
|a_0|
\leq
|\lambda_0|
\left|\frac{A_h(\lambda_0)}
{D_h(\lambda_0)}\right|.
\]
Moreover, there exists an analytic map 
$x=(a,s,p):\mathbb D\to\overline P$ such that 
\[\sup\biggl\{
|a_0|: x(0)=(0,0,0),\
x(\lambda_0)=(a_0,s_0,p_0),
(s,p)\in\mathcal H(\lambda_0;s_0,p_0)
\biggr\}=
|\lambda_0|M(\lambda_0;s_0,p_0).
\]
Equivalently, the extremal value is
\[
\begin{aligned}
|\lambda_0|
\sqrt{1-\frac14|s_0|^2}
\exp\left(
-\inf_{h=(s,p)\in\mathcal H(\lambda_0;s_0,p_0)}
\mathcal G_s(\lambda_0)
\right).
\end{aligned}
\]
The supremum is attained if and only if
the supremum defining
$M(\lambda_0;s_0,p_0)$ is attained by some
member of $\mathcal H(\lambda_0;s_0,p_0)$.
\end{prop}

\begin{proof}
Fix
$h=(s,p)\in\mathcal H(\lambda_0;s_0,p_0)$.
By Proposition~\ref{prop:corrected}, there
exists an analytic pentablock-valued
interpolant whose last two components
coincide with $h$ if and only if
\[
|a_0|
\leq
|\lambda_0|
\left|\frac{A_h(\lambda_0)}
{D_h(\lambda_0)}\right|.
\]
Furthermore, whenever this inequality
holds, the interpolant may be chosen
to be rational $\overline P$-inner.
Consequently, for each fixed $h$, the
admissible values of the first coordinate
form the closed disc
\[
\left\{
a_0\in\mathbb C:
|a_0|
\leq
|\lambda_0|
\left|\frac{A_h(\lambda_0)}
{D_h(\lambda_0)}\right|
\right\}.
\]
In particular, the largest admissible
modulus for this fixed $h$ is
\(
|\lambda_0|
\left|\frac{A_h(\lambda_0)}
{D_h(\lambda_0)}\right|.
\)(
Taking the supremum over all
$h\in\mathcal H(\lambda_0;s_0,p_0)$,
we obtain $x=(a,s,p):\mathbb D\to\overline P$ is analytic,
\[
\begin{aligned}
&\sup\biggl\{
|a_0|: x(0)=(0,0,0),\
x(\lambda_0)=(a_0,s_0,p_0),
(s,p)\in\mathcal H(\lambda_0;s_0,p_0)
\biggr\}\\
&=
|\lambda_0|
\sup_{h\in\mathcal H(\lambda_0;s_0,p_0)}
\left|\frac{A_h(\lambda_0)}
{D_h(\lambda_0)}\right|\\
&\qquad=
|\lambda_0|M(\lambda_0;s_0,p_0).
\end{aligned}
\]
By Proposition~\ref{prop:green},
\[
\left|\frac{A_h(\lambda_0)}
{D_h(\lambda_0)}\right|
=
\sqrt{1-\frac14|s_0|^2}\,
e^{-\mathcal G_s(\lambda_0)}.
\]
Since the factor
$\sqrt{1-\frac14|s_0|^2}$ is independent
of $h$, taking the supremum gives
\[
\begin{aligned}
M(\lambda_0;s_0,p_0)
&=
\sqrt{1-\frac14|s_0|^2}
\sup_{h\in\mathcal H(\lambda_0;s_0,p_0)}
e^{-\mathcal G_s(\lambda_0)}\\
&=
\sqrt{1-\frac14|s_0|^2}
\exp\left(
-\inf_{h\in\mathcal H(\lambda_0;s_0,p_0)}
\mathcal G_s(\lambda_0)
\right).
\end{aligned}
\]
Finally, suppose that the supremum of
the admissible moduli is attained by an
interpolant whose last two components
are $h_*\in\mathcal H(\lambda_0;s_0,p_0)$.
The fixed-component criterion then gives
\[
M(\lambda_0;s_0,p_0)
\leq
\left|\frac{A_{h_*}(\lambda_0)}
{D_{h_*}(\lambda_0)}\right|.
\]
The reverse inequality follows from the
definition of $M$, so $h_*$ attains the
supremum defining $M$.

Conversely, if some
$h_*\in\mathcal H(\lambda_0;s_0,p_0)$
attains this supremum, then
Proposition~\ref{prop:corrected} provides
a rational $\overline P$-inner interpolant
for every $a_0$ satisfying
\[
|a_0|
\leq
|\lambda_0|M(\lambda_0;s_0,p_0).
\]
In particular, an interpolant exists with
\(
|a_0|
=
|\lambda_0|M(\lambda_0;s_0,p_0).
\)
This proves the assertion concerning
attainment and completes the proof.
\end{proof}

\begin{rem}
\label{rem:variational-scope}
Proposition~\ref{prop:fixed-class} gives
an exact variational description of the
interpolation problem within the class
of analytic pentablock-valued maps whose
last two components form a rational
$\Gamma$-inner function.
It is important to distinguish this
restricted problem from the unrestricted
Schwarz interpolation problem for the
pentablock. Indeed, the last two
components of an arbitrary analytic map
from $\mathbb D$ into $\overline P$ need
not form a rational $\Gamma$-inner function.
Consequently, the extremal quantity
$M(\lambda_0;s_0,p_0)$ does not, by itself,
determine the full range of admissible
first-coordinate interpolation values.
Nevertheless, the variational formulation
identifies a natural extremal problem:
the maximization of the outer spectral
factor, or equivalently, the minimization
of the Green potential associated with
the first coordinate of a rational
$\Gamma$-inner interpolant.
Whether this restricted extremal problem
determines the unrestricted interpolation
problem requires further investigation.
\end{rem}
\noindent
\textbf{Open problem:}
Let $\lambda_0\in\mathbb D\setminus\{0\}$
and $x_0=(a_0,s_0,p_0)\in\overline P$.
Determine whether there exists an analytic
map
\(
x:\mathbb D\longrightarrow\overline P
\)
satisfying
\[
x(0)=(0,0,0),
\quad
x(\lambda_0)=(a_0,s_0,p_0)
\]
if and only if
\[
|s_0|<2,
\quad
\kappa(s_0,p_0)\leq|\lambda_0|,
\quad
|a_0|\leq|\lambda_0|F(s_0,p_0).
\]
The first two inequalities are necessary
by the Agler--Young Schwarz lemma for
the symmetrized bidisc, while the third
is necessary by
Proposition~\ref{prop:necessary}.
By Proposition~\ref{prop:s0zero},
these conditions are also sufficient
when $s_0=0$. Their sufficiency for
general $s_0$ remains to be determined.
A complete characterization of the
Schwarz interpolation problem for the
pentablock would also determine the
Lempert function
$\widetilde{k}_P(0,\cdot)$.
Such a characterization would provide
further insight into the relationship
between the Lempert function and the
Carath\'eodory distance on the pentablock,
a question considered by Agler, Lykova,
and Young \cite{ALY15}.

\section{AI Declarations:}
ChatGPT Go was used to assist in the development of Subsection 5.3 and to enhance the clarity, readability, and linguistic presentation of the manuscript. The authors take full responsibility for the mathematical content, accuracy, and conclusions of the work.

\end{document}